\documentclass[12pt]{amsart}
\usepackage{amscd,amssymb,epsfig}
\calclayout

\newif\ifpdf

\ifx\pdfoutput\undefined

  \pdffalse 
\else 
  \pdftrue\pdfoutput=1 
\fi 
 
\numberwithin{equation}{section} 
\def\RR{\mathbb R}

\newcommand{\dint}{\displaystyle\int}   

\newcommand{\h}[1]{\hat{#1}}
\newcommand{\ud}{\,d} 
\newcommand{\Sym}{\mathbb{S}}

\renewcommand\P{{\mathcal P}}

\renewcommand{\div}{\operatorname{div}}

\newtheorem{thm}{Theorem}[section] 
 
\newtheorem{lemma}[thm]{Lemma}

\def\whsq{\vbox to 5.8pt 
{\offinterlineskip\hrule 
\hbox to 5.8pt{\vrule height 
5.1pt\hss\vrule height 5.1pt}\hrule}} 
 
\def\Qed{{\hfill {\whsq}}} 
\def\<{\langle} 
\def\>{\rangle} 
\def\PP{{\mathop{{\rm I}\kern-.2em{\rm P}}\nolimits}} 
\def\FF{{\mathop{{\rm I}\kern-.2em{\rm F}}\nolimits}}   
\def\ZZ{{\mathop{{\rm I}\kern-.2em{\rm Z}}\nolimits}} 
 
\newlength{\sidemargin} 
\begin{document}
\title[]{
Two remarks on rectangular mixed finite elements for elasticity 
}


\author{Gerard Awanou}
\address{Department of Mathematical Sciences,
Northern Illinois University,
Dekalb, IL, 60115}
\email{awanou@math.niu.edu}
\urladdr{http://www.math.niu.edu/\char'176awanou}

\begin{abstract}
The lowest order nonconforming rectangular element in three dimensions involves
54 degrees of freedom for the stress and 12 degrees of freedom for the displacement.
With a modest increase in the number of degrees of freedom (24 for the stress),
we obtain a conforming rectangular element for linear elasticity in three dimensions.
Moreover, unlike the conforming plane rectangular or simplicial elements, this element does not involve any vertex degrees of freedom.
Second, we remark that further low order elements can be constructed by approximating 
the displacement with rigid body motions. This results in a pair of conforming elements with 72 degrees of freedom for the stress and
6 degrees of freedom for the displacement.
\end{abstract}

\maketitle

\section{Introduction}
Let $\Omega \subset \mathbb{R}^3$ be a 
contractible polygonal domain
occupied by a linearly elastic body and let $ L^2(\Omega, \mathbb{R}^3) $ be the space
of square integrable vector fields. We denote by $\Sym$ the space
of $3\times 3$ symmetric matrix fields and by 
$ H(\div,\Omega,
\Sym) $ the space
of square integrable symmetric matrix fields with divergence, taken row-wise, 
square integrable. The stress field $\sigma$ and the displacement $u$ are the unknowns in the elasticity equations
$$
A \sigma = \epsilon(u), \qquad \div \sigma = f, \quad \text{in} \, \Omega.
$$
Here $\epsilon(u)=(\partial u_i/\partial x_j + \partial u_j/\partial x_i)_{i,j=1,\ldots,3}$ 
is the symmetric part of the gradient and $f$ encodes the body forces.
The compliance tensor $A = A (x) : \Sym \to \Sym$ is given, bounded and symmetric
positive definite uniformly with respect to $x \in \Omega$.
The pair $(\sigma,u)$ is the unique solution of the Hellinger-Reissner variational 
formulation: Find $(\sigma,u) \in H(\div,\Omega,\Sym) \times L^2(\Omega, \mathbb{R}^3) $ 
such that 

\begin{align}
\begin{split} \label{cf}
\int_{\Omega} (A \sigma : \tau + \div \tau \cdot u) \ud x & = 0, \qquad \tau \in H(\div,\Omega,\Sym), \\
\int_{\Omega} \div \sigma \cdot v \ud x & = \int_{\Omega} f \cdot v \ud x, \qquad v \in L^2(\Omega, 
\mathbb{R}^3),
\end{split}
\end{align}
where 
$\sigma: \tau = \sum_{i,j=1}^3
\sigma_{ij} \tau_{ij}$ is the Frobenius product of $\sigma$ and $\tau$
and we assume $u=0$ on $\partial{\Omega}$.

For several decades, the existence of stable, conforming mixed elements for elasticity was an open problem.
Starting with the seminal work of Arnold and Winther, 
\cite{Arnold2002,Arnold2003}, for triangular elements, extended to
tetrahedral elements in \cite{Adams2005,
Arnold2008,
Guzman2010}, conforming and nonconforming mixed 
finite elements for elasticity on rectangular meshes have been constructed by 
several authors 
\cite{Arnold2005,Awanou2009, 
Yi2005, Yi2006, Hu2007/08, Man2009, Chen2010}. 
We refer to \cite{Arnold2002,Awanou2009,Arnold2007,Awanou10b}
and the references therein for alternative approaches to mixed formulations of the elasticity equations.
In particular, we note advances on mixed finite elements for elasticity where the symmetry of the stress field
is enforced weakly using Lagrange multipliers, 
\cite{Amara1979,Arnold1984a,Stenberg1986,Stenberg1988,Stenberg1988a,Morley1989,Arnold2006,Arnold2007,Falk08,Guzman10c,Guzman10b,Guzman10d,
Arnold2010b}. In general, it requires a great deal of ingenuity to construct mixed finite elements for elasticity.

The element described in this paper is the only one known, in three dimension on rectangular meshes, which has strong symmetry for the stress field and is $H(\div)$
conforming. The tetrahedral analogues have a high number of degrees of freedom and involve vertex degrees of freedom.
In general, on an arbitrary triangulation, vertex degrees of freedom are unavoidable for conforming elements for the stress
field with strong symmetry \cite{Arnold2002,Arnold2008}. Although it could be expected that on some triangulation, 
elements can be constructed with no vertex degrees of freedon \cite{Guzman10b}, it was a surprise to us that this is possible
on rectangular meshes in three dimensions which are prefered over tetrahedral elements by some engineers. 
Progress towards the discovery of the element was made possible by the observation in \cite{Chen2010} that for plane elasticity, no vertex degrees
of freedom are necessary for the normal normal component of the stress field.
Although the resulting element
is relatively complicated to use, 72 degrees of freedom for the stress or 12 degrees of freedom on average for each component of the stress field,
we believe the strong symmetry property of the element justifies the additional degrees of freedom (the lowest order nonconforming element involves
54 degress of freedom for the stress). An example of
a preference in some situations for elements with strong symmetry for two dimensional elasticity is \cite{Nicaise2008} p. 4. We quote
``Unfortunately, there exists no simple low order conforming element for \ldots''
($H(\div,\Omega,\Sym)$). ``Possible remedies are to use macro-elements or to weaken the symmetry condition, \ldots''
``However, if we weaken the symmetry, an additional factor due to the nonconformity of the approach occurs''. We claim that for the three dimensional problem, the element described in this paper is the only reasonable option available.


The paper is organized as follows: In the next section we describe the finite
element spaces on each rectangle after outlining a general framework for the construction of
mixed finite elements for elasticity. We show how one may obtain shape functions for three dimensional 
rectangular elements for elasticity from the corresponding plane elements using tensor products. In Section 3
we prove stability and convergence of the mixed finite elements using arguments now standard. We conclude with
some remarks on low order elements using 
the space of rigid body motions to approximate the displacement and the extension to higher order elements.

\section{Description of the finite elements}
We denote by 
$H^k(K, X)$ the
space of functions with domain $K \subset \RR^3$, taking values in the finite
dimensional space $X$, and with all derivatives of order at most $k$
square integrable. For our purposes, $X$ will be either 
$\Sym$, $ \RR^3, \RR^2
\mbox{ or } \RR$, and in the latter case, we simply write $H^k(X)$. The
norms in $H^k(K,X)$ and $H^k(K)$ are denoted respectively by $|| \cdot
||_{H^k}$ and $|| \cdot ||_k$. 
Let $\mathcal{T}_h$ denote a conforming partition of $\Omega$ into rectangles $K$ of size bounded by $h$, which is
quasi-uniform in the sense that the aspect ratio of the rectangles is bounded by a fixed constant.
Whether $K$ is a plane rectangle or a three dimensional brick will be clear from the context.

We first outline a set of general criteria to satisfy in order to obtain stable mixed finite elements for the linear elasticity problem.
A conforming mixed finite element approximation of \eqref{cf} consists in
choosing spaces $\Sigma_h \subset H(\div,\Omega,\Sym)$ and $V_h \subset L^2(\Omega, \mathbb{R}^3)$ and determining $(\sigma_h,u_h) \in \Sigma_h 
\times V_h$ such that:
\begin{align}
\begin{split} \label{cfD}
\int_{\Omega} (A \sigma_h : \tau_h + \div \tau_h \cdot u_h) \ud x & = 0, \qquad 
\tau_h \in \Sigma_h, \\
\int_{\Omega} \div \sigma_h \cdot v_h \ud x & = \int_{\Omega} f \cdot v_h \ud x, \qquad v_h \in V_h.
\end{split}
\end{align}
We recall that a matrix field $\tau \in  \Sigma_h$ if and 
only if $\tau n$ is continuous across the internal faces of $\mathcal{T}_h$ 
where $n$ denotes one of the unit vectors normal to the internal face. It is now known that for stability, \cite{Arnold2002,Arnold2003,Awanou2009,Arnold2008},
the following two conditions are sufficient:
\begin{itemize}
\item
$\div \Sigma_h \subset V_h$

\item
There exists a linear operator $\Pi_h : H^1 (\Omega, \Sym) \rightarrow \Sigma_h$,
such that there is a constant $c$ independent of $h$  with $|| \Pi_h
\tau ||_0 \leq  c ||\tau||_1$ for all $\tau \in H^1 (\Omega, \Sym)$, and such that
$\div \Pi_h \tau = P_h \div \tau$ for all $\tau \in H^1(\Omega, \Sym)$ where
$P_h : L^2 (\Omega, \RR^2) \rightarrow V_h$ denotes the $L^2$ projection.
\end{itemize}
If we assume that an interpolation operator $\Pi_h : H^1 (\Omega, \Sym) \rightarrow \Sigma_h$
can be defined which reproduces the degrees of freedom of $\tau \in \Sigma_h$, the second condition is satisfied
if the degrees of freedom for the stress space are chosen such that
\begin{equation}
\dint_K (\div \Pi_h \tau - \div \tau) \cdot v \, \ud x = - \dint_K
(\Pi_h \tau - \tau ) : \epsilon(v) \, d x + \int_{\partial K} (\Pi_h \tau - \tau )
n \cdot v \, \ud s,
\end{equation}
vanishes. Therefore in addition to require normal continuity of the degrees of freedom, it is sufficient that 
the degrees of freedom include
$\int_f \tau n \cdot v \ud s$ and $\int_K \tau:\epsilon(v)$, for each face $f$ of an element $K$ and each $v \in V_K$.
If in addition, 
the interpolation operator reproduces piecewise constant matrix fields, 
by a standard scaling argument,
it is bounded and the error analysis also follows from now classical arguments.

Before describing the new three dimensional conforming elements, we first note that the shape functions
for the low order three dimensional nonconforming element in \cite{Man2009}, can be obtained
from their two dimensional analogue by a tensor product construction.

\subsection{Tensor product construction} \label{tensor}
We use the usual notation of $\mathcal{P}_k(K,X)$
for the space of polynomials on $K$ with values in $X$ of total degree
less than $k$ 
and $\mathcal{P}_{k_1, k_2}(K,X)$ for the space of polynomials of
degree at most $k_1$ in $x_1$ and of degree at most $k_2$ in $x_2$. 
Similarly,  
$\mathcal{P}_{k_1, k_2, k_3}(K,X)$ denotes the space of polynomials of
degree at most $k_1$ in $x_1$, of degree at most $k_2$ in $x_2$ and 
of degree at most $k_3$ in $x_3$.
We
write $ \mathcal{P}_k $, $\mathcal{P}_{k_1, k_2}$ 
and $\mathcal{P}_{k_1, k_2, k_3}$ respectively when $X = \RR$. 
When the domain $K$ in $\RR^3$ is obvious, $\mathcal{P}_k(x_i,x_j)$ and 
$\mathcal{P}_{k_1,k_2}(x_i,x_j)$ denote respectively
the space of polynomials of total degree $k$ in the variables $x_i$ and $x_j$
and of degree at most $k_1$ in $x_i$ and at most $k_2$ in $x_j$ for $i$ not necessarily smaller than $j$.
We also use a similar notation for spaces of polynomials on a face.
We will denote by 
$
\begin{pmatrix}
\mathcal{P}_{k_1,k_2} & \mathcal{P}_{k_3,k_4} \\
\mathcal{P}_{k_3,k_4} & \mathcal{P}_{k_5,k_6}  
\end{pmatrix}_{\!\Sym}
$ 
the space of symmetric matrix fields 
$\sigma=
\begin{pmatrix}
\sigma_{11} & \sigma_{12} \\
\sigma_{21} & \sigma_{22}  
\end{pmatrix}
$ such that $\sigma_{11} \in \mathcal{P}_{k_1,k_2}$, 
$\sigma_{12} = \sigma_{2,1} \in \mathcal{P}_{k_3,k_4}$ and $\sigma_{22} 
\in \mathcal{P}_{k_5,k_6}$. Similarly the notation $\begin{pmatrix}
\mathcal{P}_{k_1,k_2} \\
\mathcal{P}_{k_3,k_4} 
\end{pmatrix}$ for $\mathcal{P}_{k_1,k_2} \times 
\mathcal{P}_{k_3,k_4}$ will also
be used. By
$
\begin{pmatrix}
\mathcal{P}_{k_1,k_2,k_3} & \mathcal{P}_{k_4,k_5,k_6} & \mathcal{P}_{k_7,k_8,k_9} \\
 \mathcal{P}_{k_4,k_5,k_6} &  \mathcal{P}_{k_{10},k_{11},k_{12}} &  
\mathcal{P}_{k_{13},k_{14},k_{15}} 
\\
\mathcal{P}_{k_7,k_8,k_9} & \mathcal{P}_{k_{13},k_{14},k_{15}} & 
\mathcal{P}_{k_{16},k_{17},k_{18}}
\end{pmatrix}_{\!\Sym}
$
we mean the space of symmetric matrix fields $\sigma=
\begin{pmatrix}
\sigma_{11} & \sigma_{12} &  \sigma_{13} \\
\sigma_{21} & \sigma_{22} &  \sigma_{23} \\
\sigma_{31} & \sigma_{32} &  \sigma_{33}
\end{pmatrix}
$ such that $\sigma_{11} \in \mathcal{P}_{k_1,k_2,k_3}$,
$\sigma_{22}
\in \mathcal{P}_{k_7,k_8,k_9}$, $\sigma_{33}
\in \mathcal{P}_{k_{13},k_{14},k_{15}}$,
$\sigma_{12} = \sigma_{21} \in \mathcal{P}_{k_4,k_5,k_6}$, $\sigma_{13}=\sigma_{31}
\in \mathcal{P}_{k_7,k_8,k_9}$ and $\sigma_{23}=\sigma_{32}
\in \mathcal{P}_{k_{13},k_{14},k_{15}}$.
We will also write $\begin{pmatrix}
\mathcal{P}_{k_1,k_2,k_3} \\
\mathcal{P}_{k_4,k_5,k_6} \\
\mathcal{P}_{k_7,k_8,k_9}
\end{pmatrix}$ in place of $\mathcal{P}_{k_1,k_2,k_3} \times
\mathcal{P}_{k_4,k_5,k_6} \times \mathcal{P}_{k_7,k_8,k_9}$.
Often at the place of the spaces $\mathcal{P}_{k, l, m}(K,X)$,
we will write explicit vector spaces as span.

For two finite dimensional spaces $E$ and $F$, we denote by $E \otimes F$ the vector space equal to
$\mbox{span} \{e f\}$, $e \in E, f \in F$. A low order plane nonconforming element described in \cite{Hu2007/08}
uses as shape functions on an element $K$ for the stress 
$$
\sigma \in \Sigma_{\text{ncf}}^{2\text{D}}(K) = \begin{pmatrix}
P_{1,1} & \mbox{span}\{1, x_1, x_2, x_1^2, x_2^2\} \cr
\mbox{span}\{1, x_1, x_2, x_1^2, x_2^2\} & P_{1,1}
\end{pmatrix}_{\Sym},
$$
and for the displacement
$v \in V^{2\text{D}}(K) = \begin{pmatrix}
\mbox{span}\{1, x_2\} \cr
\mbox{span}\{1, x_1\}
\end{pmatrix}$.
We postulate that for the three dimensional generalization, the choice of finite dimensional spaces for each component
of the stress field $\begin{pmatrix}
\sigma_{11} & \sigma_{12} & \sigma_{13} \cr
\sigma_{12} & \sigma_{22} & \sigma_{23} \cr
\sigma_{13} & \sigma_{23} & \sigma_{33}
\end{pmatrix}$ and each component of the displacement 
$\begin{pmatrix}v_{1} \cr v_2 \cr v_3  \end{pmatrix}$ should conform to the following requirements:
\begin{itemize}
\item $\begin{pmatrix} \sigma_{11} & \sigma_{12} \cr \sigma_{12} & \sigma_{22} \end{pmatrix}$
and $\begin{pmatrix}v_{1} \cr v_2 \end{pmatrix}$ belong to the plane elements in the variables $x_1,x_2$ when restricted to the planes $x_3=0,1$,
\item   $\begin{pmatrix} \sigma_{11} & \sigma_{13} \cr \sigma_{13} & \sigma_{33} \end{pmatrix}$
and $\begin{pmatrix}v_{1} \cr v_3 \end{pmatrix}$ belong to the plane elements 
in the variables $x_1,x_3$ when restricted to the planes $x_2=0,1$,
\item   $\begin{pmatrix} \sigma_{22} & \sigma_{23} \cr \sigma_{23} & \sigma_{33} \end{pmatrix}$
and $\begin{pmatrix}v_{2} \cr v_3 \end{pmatrix}$ belong to the plane elements 
in the variables $x_2,x_3$ when restricted to the planes $x_1=0,1$.
\end{itemize}
We note that the above statements can be summarized as follows: for $i<j, i,j \in \{1,2,3\}, k \notin \{i,j\} $,
$\begin{pmatrix} \sigma_{ii} & \sigma_{ij} \cr \sigma_{ij} & \sigma_{jj} \end{pmatrix}$
and $\begin{pmatrix}v_{i} \cr v_j \end{pmatrix}$ belong to the plane elements 
in the variables $x_i,x_j$ when restricted to the planes $x_k=0,1$, that is
$\sigma_{ii} \in \P_{1,1}(x_i,x_j), \sigma_{i,j} \in \mbox{span}\{1, x_i, x_j, x_i^2, x_j^2\}$
and $v_i \in \mbox{span}\{1, x_j\}, v_j \in \mbox{span}\{1, x_i\}$ on the planes $x_k=0,1$.

A set of shape functions for the three dimensional elasticity problem is obtained by taking as finite element spaces
for the components of the stress and the displacement, tensor products of the plane finite element spaces in the variables $x_i,x_j$
with \mbox{span}$\{1,x_k\}$. For this note that for $i<j, i,j \in \{1,2,3\}, k \notin \{i,j\} $,
\begin{align*}
\begin{split}
\P_{1,1}(x_i,x_j) \otimes  \mbox{span} \{1,x_k\} & = \P_{1,1,1},  \\
\mbox{span}\{1, x_i, x_j, x_i^2, x_j^2\} \otimes \mbox{span} \{1,x_k\} & = 
\mbox{span}\{1, x_i, x_j, x_i^2, x_j^2, x_k, \\
\qquad \qquad & x_i x_k, x_j x_k, x_i^2 x_k, x_j^2x_k \}
\\
\mbox{span} \{1,x_i\}  \otimes  \mbox{span} \{1,x_k\}, 
&= \mbox{span} \{1,x_i,x_k, x_ix_k\} =\P_{1,1}(x_i,x_k).
\end{split}
\end{align*}
We thus obtain the shape functions for the low order three dimensional element described in \cite{Man2009}. That is
if we denote by $\Sigma_{\text{ncf}}^{3\text{D}}(K)$ the finite element for the stress and $V^{3\text{D}}(K)$ the finite element space for the displacement,
$\sigma \in \Sigma_{\text{ncf}}^{3\text{D}}(K)$ if and only if $\sigma_{ii} \in \P_{1,1,1}, i=1,2,3$, $\sigma_{i,j} \in 
\mbox{span}\{1, x_i, x_j, x_i^2, x_j^2, x_k,$ $ x_i x_k, x_j x_k, x_i^2 x_k, x_j^2x_k \}$ 
and $v \in V^{3\text{D}}(K)$ if and only if $v_i \in \P_{1,1}(x_j,x_k)$ for $i<j, i,j \in \{1,2,3\}, k \notin \{i,j\}$.
Alternatively, if we define 
$$
\tilde{\P}_2(x_i,x_j) = \mbox{span}\{1, x_i, x_j, x_i^2, x_j^2, x_k, x_i x_k, x_j x_k, x_i^2 x_k, x_j^2x_k \},
$$
then
$$
\Sigma_{\text{ncf}}^{3\text{D}}(K) = \begin{pmatrix}
\P_{1 1 1} & \tilde{\P}_2(x_1,x_2) & \tilde{\P}_2(x_1,x_3) \cr
\tilde{\P}_2(x_1,x_2) & \P_{1 1 1} & \tilde{\P}_2(x_2,x_3) \cr
\tilde{\P}_2(x_1,x_3) & \tilde{\P}_2(x_2,x_3) & \P_{111}
\end{pmatrix}_{\Sym}.
$$
\subsection{Description of the elements}
We first describe the shape functions of a conforming rectangular element recently introduced in \cite{Chen2010}.
Then, following the strategy described in Section \ref{tensor}, we give a set of shape functions for the
finite element spaces in three dimensions. We give a more concise
description of the stress finite element space
before giving a unisolvent set of degrees of freedom.
Let us denote by $\Sigma_{\text{cf}}^{2\text{D}}(K)$ the space of shape functions for the stress for the plane conforming element and
by $\Sigma_{\text{cf}}^{3\text{D}}(K)$ the corresponding three dimensional element. Then \cite{Chen2010},
\begin{align*}
\begin{split}
\Sigma_{\text{cf}}^{2\text{D}}(K)& =\Sigma_{\text{ncf}}^{2\text{D}}(K) + \mbox{span}
\bigg\{
\begin{pmatrix}
-\frac{1}{2}x_1^2 & x_1 x_2 \cr
x_1 x_2 & -\frac{1}{2}x_2^2
\end{pmatrix}, \begin{pmatrix}
-\frac{1}{3}x_1^3 & x_1^2 x_2 \cr
x_1^2 x_2 & -x_1 x_2^2
\end{pmatrix}, \begin{pmatrix}
-x_1^2 x_2 & x_1x_2^2 \cr
x_1x_2^2 & -\frac13 x_2^3
\end{pmatrix}, \\
& \qquad \qquad \qquad \qquad \qquad \qquad \qquad \qquad \qquad \begin{pmatrix}
-\frac23 x_1^3 x_2 & x_1^2 x_2^2 \cr
x_1^2 x_2^2 & -\frac23 x_1x_2^3
\end{pmatrix}
\bigg\}.
\end{split}
\end{align*}
We choose as finite element space for the displacement the same space $V^{3\text{D}}(K)$ 
used in the nonconforming case.
Note that it can equivalently be written as
$\P_{011} \times \P_{101} \times \P_{110}$ 
and will be denoted by $V(K)$ in the remaining part of this paper.
We define $\Sigma_{\text{cf}}^{3\text{D}}(K)$ as the direct sum of $\Sigma_{\text{ncf}}^{3\text{D}}(K)$ 
and the span of 24 additional shape functions 
$\sigma^{ij,(k)}, i < j, i,j=1,2,3$ and $k=1,\ldots,8$
obtained from the additional shape functions in the definition of $\Sigma_{\text{cf}}^{2\text{D}}(K)$
in the variables $x_i$ and $x_j$ and the tensor product construction outlined in Section \ref{tensor}.
For a given matrix, we will refer to the entry on the $i$th row and $j$th column as the $(i,j)$ component.
\begin{align*}
\begin{split}
\sigma^{12,(1)} & =\begin{pmatrix}
-\frac12 x_1^2 & x_1 x_2 & 0 \cr
x_1 x_2 & -\frac12 x_2^2 & 0 \cr
0 & 0 & 0
\end{pmatrix}, \,
\sigma^{12,(2)}=\begin{pmatrix}
-\frac12 x_1^2 x_3 & x_1 x_2 x_3 & 0 \cr
x_1 x_2 x_3 & -\frac12 x_2^2 x_3 & 0 \cr
0 & 0 & 0
\end{pmatrix}, \\
  \sigma^{12,(3)}& =\begin{pmatrix}
-\frac13 x_1^3 & x_1^2 x_2 & 0 \cr
x_1^2 x_2 & -x_1 x_2^2 & 0 \cr
0 & 0 & 0
\end{pmatrix}, \,
\sigma^{12,(4)}=\begin{pmatrix}
-\frac13 x_1^3 x_3 & x_1^2 x_2 x_3 & 0 \cr
x_1^2 x_2 x_3 & -x_1 x_2^2 x_3 & 0 \cr
0 & 0 & 0
\end{pmatrix}, 
\end{split}
\end{align*}
\begin{align*}
\begin{split}
\sigma^{12,(5)} &  =\begin{pmatrix}
-x_1^2 x_2 & x_1 x_2^2 & 0 \cr
x_1 x_2^2 & -\frac13 x_2^3 x_3 & 0 \cr
0 & 0 & 0
\end{pmatrix}, \,
 \sigma^{12,(6)}=\begin{pmatrix}
-x_1^2 x_2 x_3 & x_1 x_2^2 x_3 & 0 \cr
x_1 x_2^2 x_3 & -\frac13 x_2^3 x_3 & 0 \cr
0 & 0 & 0
\end{pmatrix}, \\
\sigma^{12,(7)} & =\begin{pmatrix}
-\frac23 x_1^3 x_2 & x_1^2 x_2^2 & 0 \cr
x_1^2 x_2^2 & -\frac23 x_1 x_2^3 & 0 \cr
0 & 0 & 0
\end{pmatrix}, \,
\sigma^{12,(8)}=\begin{pmatrix}
-\frac23 x_1^3 x_2 x_3 & x_1^2 x_2^2 x_3 & 0 \cr
x_1^2 x_2^2 x_3 & -\frac23 x_1 x_2^3 x_3 & 0 \cr
0 & 0 & 0
\end{pmatrix}. 
\end{split} 
\end{align*}
Note that the third column of all the above matrices is the zero vector. We conclude that for $k=1,\ldots,8$, $\sigma^{12,(k)}$ has its $(1,3)$
and $(2,3)$ component 0, but the $(1,2)$ component non-zero.
\begin{align*}
\begin{split}
\sigma^{13,(1)} & =
\begin{pmatrix}
-\frac12 x_1^2 & 0 & x_1 x_3 \cr
0 & 0 & 0 \cr
x_1 x_3 & 0 & -\frac12 x_3^2
\end{pmatrix}, \,
\sigma^{13,(2)}=\begin{pmatrix}
-\frac12 x_1^2 x_2 & 0 & x_1 x_2 x_3 \cr
0 & 0 & 0 \cr
x_1 x_2 x_3 & 0 & -\frac12 x_2 x_3^2
\end{pmatrix}, \\
\sigma^{13,(3)}& =\begin{pmatrix}
-\frac13 x_1^3 & 0 & x_1^2 x_3 \cr
0 & 0 & 0 \cr
x_1^2 x_3 & 0 & -x_1 x_3^2
\end{pmatrix}, \,
\sigma^{13,(4)}= \begin{pmatrix}
-\frac13 x_1^3 x_2 & 0 & x_1^2 x_2 x_3 \cr
0 & 0 & 0 \cr
x_1^2 x_2 x_3 & 0 & x_1 x_2 x_3^2
\end{pmatrix}, \\
\sigma^{13,(5)} &=\begin{pmatrix}
-x_1^2 x_3 & 0 & x_1 x_3^2 \cr
0 & 0 & 0 \cr
x_1 x_3^2 & 0 & -\frac13 x_3^3
\end{pmatrix}, \,
\sigma^{13,(6)}=\begin{pmatrix}
-x_1^2 x_2 x_3 & 0 & x_1 x_2 x_3^2 \cr
0 & 0 & 0 \cr
x_1 x_2 x_3^2 & 0 & -\frac13 x_2 x_3^3
\end{pmatrix}, \\
\sigma^{13,(7)} &= \begin{pmatrix}
-\frac23 x_1^3 x_3 & 0 & x_1^2 x_3^2 \cr
0 & 0 & 0 \cr
x_1^2 x_3^2 & 0 & -\frac23 x_1 x_3^3
\end{pmatrix},
\sigma^{13,(8)}=\begin{pmatrix}
-\frac23 x_1^3 x_2 x_3 & 0 & x_1^2 x_2 x_3^2 \cr
0 & 0 & 0 \cr
x_1^2 x_2 x_3^2 & 0 & -\frac23 x_1 x_2 x_3^3
\end{pmatrix}.
\end{split} 
\end{align*}
Note that the second row or second column of all the matrices $\sigma^{13,(k)}, k=1,\ldots,8$ is the zero vector. We conclude that for 
$k=1,\ldots,8$, $\sigma^{13,(k)}$ has its $(1,2)$
and $(2,3)$ component 0, but the $(1,3)$ component non-zero.
\begin{align*}
\begin{split}
\sigma^{23,(1)} &=\begin{pmatrix}
0 & 0 & 0 \cr
0 & -\frac12 x_2^2 & x_2 x_3 \cr
0 & x_2 x_3 & -\frac12 x_3^2
\end{pmatrix}, \,
\sigma^{23,(1)} = \begin{pmatrix}
0 & 0 & 0 \cr
0 & -\frac12 x_1 x_2^2 & x_1 x_2 x_3 \cr
0 & x_1 x_2 x_3 & -\frac12 x_1 x_3^2
\end{pmatrix},\\
\sigma^{23,(3)} &=\begin{pmatrix}
0 & 0 & 0 \cr
0 & -\frac13 x_2^3 & x_2^2 x_3 \cr
0 & x_2^2 x_3 & -x_2 x_3^2
\end{pmatrix}, \,
\sigma^{23,(4)} =\begin{pmatrix}
0 & 0 & 0 \cr
0 & -\frac13 x_1 x_2^3 & x_1 x_2^2 x_3 \cr
0 & x_1 x_2^2 x_3 & -x_1 x_2 x_3^2
\end{pmatrix}, 
\end{split}
\end{align*}
\begin{align*}
\begin{split}
\sigma^{23,(5)} &= \begin{pmatrix}
0 & 0 & 0 \cr
0 & -x_2^2 x_3 & x_2 x_3^2 \cr
0 & x_2 x_3^2 & -\frac13 x_3^3
\end{pmatrix},\,
\sigma^{23,(6)} =\begin{pmatrix}
0 & 0 & 0 \cr
0 & -x_1 x_2^2 x_3 & x_1 x_2 x_3^2 \cr
0 & x_1 x_2 x_3^2 & -\frac13 x_1 x_3^3
\end{pmatrix}, \\
\sigma^{23,(7)} & =\begin{pmatrix}
0 & 0 & 0 \cr
0 & -\frac23 x_2^3 x_3 & x_2^2 x_3^2 \cr
0 & x_2^2 x_3^2 & -\frac23 x_2 x_3^3
\end{pmatrix}, \,
\sigma^{23,(8)} = \begin{pmatrix}
0 & 0 & 0 \cr
0 & -\frac23 x_1 x_2^3 x_3 & x_1 x_2^2 x_3^2 \cr
0 & x_1 x_2^2 x_3^2 & -\frac23 x_1 x_2 x_3^3
\end{pmatrix}.
\end{split}
\end{align*}
Note that the first row or first column of all the matrices $\sigma^{23,(k)}, k=1,\ldots,8$ is the zero vector. We conclude that for 
$k=1,\ldots,8$, $\sigma^{23,(k)}$ has its $(1,2)$
and $(1,3)$ component 0, but the $(1,3)$ component non-zero.

We can summarize the above observations by stating that 
\begin{equation}
\sigma^{ij,(k)}_{pq}, k=1,\ldots,8=0 \, \text{for} \, i<j \, \text{unless} \, (p,q)=(i,j). \label{observe}
\end{equation}
Note that \eqref{observe} can also be seen directly from the tensor product construction.
We also note that like their two dimensional analogue, the above additional shape functions are divergence free and that
the dimension of $\Sigma_{\text{cf}}^{3\text{D}}(K)$ is $54+24=78$. 
This is implicit in the tensor product construction and also follows from Lemma \ref{lemdof} below.
We can describe $\Sigma_{\text{cf}}^{3\text{D}}(K)$
in a more concise form as follows: $\tau \in \Sigma_{\text{cf}}^{3\text{D}}(K)$ if and only if
\begin{equation}
\tau \in \Sigma(K):=\bigg\{\begin{pmatrix} \P_{311} & \P_{2 2 1} & \P_{212} \cr
\P_{2 2 1} & \P_{131} & \P_{122} \cr
\P_{212} & \P_{122} & \P_{113}
 \end{pmatrix}_{\Sym}, \div \tau \in \begin{pmatrix} \P_{011} \cr \P_{101} \cr \P_{110}  \end{pmatrix}\bigg\}.
\end{equation}
Clearly $\Sigma_{\text{cf}}^{3\text{D}}(K)$ is contained in the latter space. Notice that the dimension of
$S_2:=\begin{pmatrix} \P_{311} & \P_{2 2 1} & \P_{212} \cr
\P_{2 2 1} & \P_{131} & \P_{122} \cr
\P_{212} & \P_{122} & \P_{113}
 \end{pmatrix}$ is 102. And for $\tau \in S_2, \div \tau \in \begin{pmatrix} \P_{211} \cr \P_{121} \cr \P_{112}  \end{pmatrix}$
 which has dimension 36. Hence the constraint $\div \tau \in V(K)$
imposes 36-12=24 constraints. It follows that the dimension
 of $\Sigma(K)$ is at least 78. That its dimension is 78 follows from the unisolvency of the degrees of freedom we give in Lemma \ref{lemdof}
 below. Hence, $\Sigma(K)=\Sigma_{\text{cf}}^{3\text{D}}(K)$.
 
 Before we give the degrees of freedom of $\Sigma_{\text{cf}}^{3\text{D}}(K)$, we explain 
 why unlike the conforming plane rectangular or simplicial elements, this element does not involve any vertex degrees of freedom.
 In general, on an arbitrary simplicial triangulation, vertex degrees of freedom are necessary \cite{Arnold2002,Arnold2008}. It will be enough to consider the unit cube $[0,1]^3$.
 We first write explicitly the condition of normal continuity of
 $$
\sigma = \begin{pmatrix}
\sigma_{11} & \sigma_{12} & \sigma_{13} \cr
\sigma_{12} & \sigma_{22} & \sigma_{23} \cr
\sigma_{13} & \sigma_{23} & \sigma_{33}
\end{pmatrix}.
$$
Clearly on
horizontal faces $x_3 = 0,1$, 
$\sigma n  = \pm(\sigma_{13}, \sigma_{23}, \sigma_{33})$.
On
vertical faces $x_1 = 0,1$, 
$\sigma n = \pm (\sigma_{11}, \sigma_{12}, \sigma_{13})$
and on 
lateral faces $x_2 = 0,1$
$\sigma n = \pm (\sigma_{12}, \sigma_{22}, \sigma_{23})$.
We therefore need
$\sigma_{33}$ continuous on faces $x_3 = 0,1$,
$\sigma_{22}$ continuous on faces $x_2 = 0,1$ and
$\sigma_{11}$ continuous on faces $x_1 = 0,1$.
Since for $i=1,2,3$ the faces $x_i=0$ and $x_i=1$ do not meet, no edge or vertex degrees of freedom are necessary
for $\sigma_{ii}$. Next, we need
$\sigma_{23}$ continuous on faces $x_3 = 0,1$, $x_2 = 0,1$, 
so we need $\sigma_{23}$ continuous on the edges $x_2 = 0$, $x_3 = 0$; $x_2 = 1$, $x_3
= 0$; $x_2 = 0, x_3 = 1$; $x_2 = 1$, $x_3 = 1$.
And similar conditions for $\sigma_{13}$ and $\sigma_{12}$, that is
for $i<j, i,j \in \{1,2,3\}, k \notin \{i,j\}$ we need $\sigma_{ij}$ continuous on the faces $x_i=0,1$ and $x_j=0,1$.
This requires the continuity of $\sigma_{ij}$ on the four edges $x_i,x_j \in \{0,1\}$.
No vertex degrees of freedom are therefore necessary here.

Let $\hat{K}=[0,1]^3$ be the reference cube and let $F: \hat{K} 
\to K$
be an affine mapping onto $K$, $F(\hat{x}) = B \hat{x} +b$, with $b=(b_1,b_2,b_3)^T$ 
and
$$
B= \begin{pmatrix}
h_1 & 0 & 0\\
0 & h_2 & 0\\
0& 0& h_3
\end{pmatrix}.
$$
To maintain the analogy with the simplicial case, although $B$ is symmetric we will write $B^T$. 
Given a matrix field $\hat{\tau}: \h{K} \to \Sym$, define $\tau: K \to 
\Sym$
by the matrix Piola transform
$$
\tau(x) = B \h{\tau}(\h{x}) B^T.
$$
It is not difficult to verify that $\tau \in \Sigma(K)$ if and only if $\h{\tau} \in \Sigma (\h{K})$.
We recall that an edge of $K$ is given by $x_i,x_j \in \{0,1\}$ and we denote by 
 $\dint_{\{x_i = 0, 1, \ x_j = 0, 1\}}$ an integral over such an edge.
We now give the degrees of freedom on $\Sigma(\h{K})$. 
\begin{lemma} \label{lemdof}
An element $\h{\sigma} \in \Sigma(\h{K})$ is uniquely determined by the following degrees of freedom
\begin{enumerate}
\item $\dint_{\{x_i = 0, 1, \ x_j = 0, 1\}} \h{\sigma}_{ij} \h{v} \, \ud x_f$, $ \h{v} \in
\mbox{span}\{1, x_k,\}$ for $k \notin \{i, j\}$, $i < j$ $i,j=1,2,3,$ (24 degrees of freedom),
\item $\dint_{x_i=0,1} \h{\sigma}_{ii} \h{v} \ud x, \h{v} \in \P_{1,1}(x_l,x_k), i\notin \{l,k\}, i=1,2,3,$ (24 degrees of freedom),
\item $\dint_{x_l = 0,1} \h{\sigma}_{ij} \h{v} \, \ud x_f$, $\h{v} \in \mbox{span}\{1, x_k\}$,
$k \notin \{i, j\}, i< j$, $l = i, j$, (24 degrees of freedom),
\item $\dint_K \h{\sigma}_{ij} \h{v} \, \ud x$, $\h{v} \in \mbox{span}\{1, x_k\}$, $k \notin \{i,
j\}, i<j, i,j=1,2,3$, (6 degrees of freedom).
\end{enumerate}
\end{lemma}

{\it Proof.}
We assume that all degrees of freedom vanish. 
Note that on the faces $x_i = 0,1$, $\h{\sigma}_{ij} \in \P_{2,1} (x_j, x_k)$, $k \notin
\{i, j\}$, $i < j, i,j=1,2,3$, (e.g. on the faces $x_2=0,1, \h{\sigma}_{23} \in \P_{2,1}(x_3,x_1)$)
and on an edge $x_i,x_j \in \{0,1\}$, $\h{\sigma}_{ij} \in \P_1(x_k)$.
By the first set of degrees of freedom, we have $\h{\sigma}_{ij} = 0$ on such an edge.
Thus on a face $x_i = 0,1$, $\h{\sigma}_{ij} = x_j(1 - x_j) \tilde{\sigma}_{ij}$, $\tilde{\sigma}_{ij} \in
\P_{0,1}(x_j, x_k)$.
Since $\dint_{x_i = 0,1} \h{\sigma}_{ij} \h{v}\, dx_f = 0$, $\h{v} \in \mbox{span}\{1,x_k\}$ by the third set of degrees of freedom, 
$\h{\sigma}_{ij} = 0$ on the faces $x_i = 0,1$.
Similarly, on the faces $x_j = 0,1$, $\h{\sigma}_{ij} \in \P_{2,1} (x_i, x_k)$, $k \notin
\{i, j\}$, $i < j, i,j=1,2,3$ and by the same argument,
$\h{\sigma}_{ij} = 0$ as well on the faces $x_j = 0,1$.

We therefore have $\h{\sigma}_{ij} = x_j(1 - x_j)x_i(1 - x_i) \tilde{\sigma}_{ij}$.
Since $\h{\sigma}_{ij} \in \P_{2,2,1}(x_i, x_j, x_k)$, we have $\tilde{\sigma}_{ij}
\in \P_{0,0,1}$.

With the last set of degrees of freedom, we conclude that
$\h{\sigma}_{ij} = 0, i<j, i,j=1,2,3$. 

Let us write
\begin{align*}
\h{\sigma} & = \tilde{\sigma}+\sum_{k=1}^8 \alpha^{12,(k)} \sigma^{12,(k)} + \alpha^{13,(k)} \sigma^{13,(k)} +\alpha^{23,(k)} \sigma^{23,(k)},
\\ 
&\qquad \qquad \quad \alpha^{ij,(k)} \in \mathbb{R}, i<j, i,j=1,2,3 \, \text{and} \, k=1,\ldots,8, \tilde{\sigma} \in \Sigma_{\text{cf}}^{3\text{D}}(K),
\end{align*} 
and recall that $\h{\sigma}_{12}=\h{\sigma}_{13}=\h{\sigma}_{23}=0$. We show that $\alpha^{12,(k)}=\alpha^{13,(k)}=\alpha^{23,(k)}=0, k=1,\ldots,8.$
This is a consequence of the tensor product construction  as the following argument shows.
By \eqref{observe}, for $i<j$
$$
\h{\sigma}_{ij} = \tilde{\sigma}_{ij}+\sum_{k=1}^8 \alpha^{ij,(k)} \sigma^{ij,(k)}_{ij}.
$$
Explicitly
\begin{align*}
\begin{split}
\h{\sigma}_{12} & = \tilde{\sigma}_{12}+\alpha^{12,(1)} x_1 x_2 + \alpha^{12,(2)} x_1 x_2 x_3 + \alpha^{12,(3)} x_1^2x_2 
+ \alpha^{12,(4)} x_1^2 x_2 x_3 + \alpha^{12,(5)} x_1 x_2^2 + \alpha^{12,(6)} x_1 x_2^2 x_3 \\
& \qquad \qquad \quad + \alpha^{12,(7)} x_1^2 x_2^2
+ \alpha^{12,(8)} x_1^2 x_2^2 x_3, \\
\h{\sigma}_{13} & = \tilde{\sigma}_{13}+\alpha^{13,(1)} x_1 x_3 + \alpha^{13,(2)} x_1 x_2 x_3 + \alpha^{13,(3)} x_1^2x_3 
+ \alpha^{13,(4)} x_1^2 x_2 x_3 + \alpha^{13,(5)} x_1 x_3^2 + \alpha^{13,(6)} x_1 x_2 x_3^2 \\
& \qquad \qquad \quad + \alpha^{13,(7)} x_1^2 x_3^2
+ \alpha^{13,(8)} x_1^2 x_2 x_3^2,\\
\h{\sigma}_{23} & = \tilde{\sigma}_{23}+\alpha^{23,(1)} x_2 x_3 + \alpha^{23,(2)} x_1 x_2 x_3 + \alpha^{23,(3)} x_2^2x_3 
+ \alpha^{23,(4)} x_1 x_2^2 x_3 + \alpha^{23,(5)} x_2 x_3^2 + \alpha^{23,(6)} x_1 x_2 x_3^2 \\
& \qquad \qquad \quad + \alpha^{23,(7)} x_2^2 x_3^2
+ \alpha^{23,(8)} x_1 x_2^2 x_3^2,
\end{split}
\end{align*}
and recall that by definition
\begin{align*}
\begin{split}
\tilde{\sigma}_{12} & \in \, \text{span} \{ 1, x_1, x_2, x_1^2, x_2^2, x_3, x_1 x_3, x_2 x_3, x_1^2 x_3, x_2^2 x_3\}, \\
\tilde{\sigma}_{13} & \in \, \text{span}\{ 1, x_1, x_3, x_1^2, x_3^2, x_2, x_1 x_2, x_2 x_3, x_1^2 x_2, x_2 x_3^2\}, \\
\tilde{\sigma}_{23} & \in \, \text{span}\{ 1, x_2, x_3, x_2^2, x_3^2, x_1, x_1 x_2, x_1 x_3, x_1 x_2^2, x_1 x_3^2\},
\end{split}
\end{align*}
and the result follows since each of the monomial appearing in the expansion
of $\h{\sigma}_{ij}, i<j$, appears only one time and for each expansion, they are all different.

We conclude that $\h{\sigma} \in \Sigma_{\text{ncf}}^{3\text{D}}(K)$.
By the second set of degrees of freedom, 
since $\h{\sigma}_{ii} \in \P_{111}$, $\h{\sigma}_{ii} = 0$ on each face $x_i = 0,1$.
Therefore, $\h{\sigma}_{ii} = x_i(1 - x_i)\tilde{\sigma}_{ii}$ for some polynomial
$\tilde{\sigma}_{ii}$ so $\h{\sigma}_{ii}=0$ for all $i=1,2,3$. \Qed

If we denote by $n_l, l=1,2,3$ the normal in the $l$th direction, each face is given by $x_i=0,1$ for some $i \in \{1,2,3\}$,
an edge is the intersection of two faces $x_i=0,1, x_j=0,1$ for some $i,j$. We can then rewrite the degrees of freedom as follows
\begin{enumerate}
\item[(1')] $\int_{\h{e}} \h{\sigma} n_i \cdot n_j \, \h{v} \, \ud x$, $ \h{v} \in \P_1(\h{e})$, for each edge with normals $n_i$ and $n_j$,
$i \neq j, i,j=1,2,3$,
\item[(2')] $\int_{\h{f}} \h{\sigma} n \cdot n \, \h{v} \ud x, \h{v} \in \P_{1,1}(\h{f})$,
\item[(3')] $\int_{\h{f}} \h{\sigma} n_i \cdot n_l \, \h{v} \ud x,$, $\h{v} \in \mbox{span}\{1, x_k\}$,
$k \notin \{i, l\}$, for each face $\h{f}$, with normal $n_i$ and $i \neq l$ and all $l=1,2,3$.  
\end{enumerate}
Note that for $\h{v} \in V(\h{K}), \epsilon(\h{v}) \in \begin{pmatrix}
0 & \mbox{span}\{1,x_3\} & \mbox{span}\{1,x_2\} \cr
\mbox{span}\{1,x_3\}  & 0 & \mbox{span}\{1,x_1\} \cr
\mbox{span}\{1,x_2\} & \mbox{span}\{1,x_1\} & 0
\end{pmatrix}_{\Sym}$, where the zero vector space is simply denoted 0. We can thus write the fourth set of degrees of freedom as
\begin{enumerate}
\item[(4')] $\int_{\h{K}} \h{\sigma}:\epsilon(\h{v}), \ \h{v} \in V(\h{K})$. 
\end{enumerate}

Given the discussion before the above lemma, the degrees of freedom in Lemma \ref{lemdof} imply that for $\h{\sigma} \in \Sigma (\h{K})$,
$\h{\sigma} n$ is uniquely determined by the degrees of freedom on the face with normal $n$.

On the rectangular partition $\mathcal{T}_h$, 
we define $\Sigma_h$ as the space of matrix
fields which belong piecewise to $\Sigma(K)$ subject to the continuity
conditions that $\tau n$ for $\tau \in \Sigma_h$ is continuous across
faces. 
We denote by $V_h$ the space of vector fields which belongs to $V_K$ for
each rectangle $K \in \mathcal{T}_h$.

\section{Stability and Error analysis}
Our finite element spaces satisfy $\Sigma_h \subset H(\div,\Sym$)
and $V_h \subset L^2(\Omega, \mathbb{R})$ with $\div \Sigma_h \subset V_h$.
We next define an interpolation operator $\Pi_h : H^1 (\Omega,\Sym) \rightarrow \Sigma_h$ with the required properties.
Because of the edge degrees of freedom, the canonical interpolation operator for $\Sigma_h$ is not well-defined on
$H^1(\Omega,\Sym)$. As in \cite{Arnold2005}, 
we first define an interpolation operator $\Pi^0_{\h{K}}: H^1(\h{K},\Sym) \to \Sigma(\h{K})$ by
\begin{alignat*}{2}
\int_{\h{e}} (\Pi^0_{\h{K}}(\h{\sigma})) n_i \cdot n_j \, \h{v} \, \ud x & = 0 \quad && \forall \h{v} \in \P_1(\h{e}) \, 
\, \text{for} \, \h{e} \, \text{with normals} \, n_i \, \text{and} \, n_j, \, i \neq j, i,j=1,2,3, \\
\int_{\h{f}} (\h{\sigma}-\Pi^0_{\h{K}}(\h{\sigma})) n \cdot n \, \h{v} \ud x & = 0 \quad && \forall \h{v} \in \P_{1,1}(\h{f}), \\
\int_{\h{f}} (\h{\sigma}-\Pi^0_{\h{K}}(\h{\sigma})) n_i \cdot n_l \ud x & = 0 \quad &&  \text{for each face} \ \h{f}, \text{with normal $n_i$ and $i \neq l$ and all $l=1,2,3$},\\
\int_{\h{K}} (\h{\sigma}-\Pi^0_{\h{K}}(\h{\sigma})):\epsilon(\h{v}) & = 0 \quad && \forall \h{v} \in V(\h{K}).
\end{alignat*}
By Lemma \ref{lemdof}, $\Pi^0_{\h{K}}$ is well defined.
Next define $\Pi^0_K: H^1(K,\Sym) \to \Sigma(K)$ by
$$
\Pi^0_K \tau(x) = B \Pi^0_{\h{K}} \h{\tau}(\h{x}) B^T,
$$
for each rectangle $K$ of $\mathcal{T}_h$ and define $\Pi^0_h:
H^1(\Omega,\Sym) \to \Sigma_h$
by
$$
(\Pi^0_h \tau)|_K = \Pi_K \tau.
$$
A standard scaling argument gives
\begin{equation}
\|\Pi^0_{h} \tau -\tau\|_0 \leq c h \|\tau\|_1 \label{pibound},
\end{equation}
where $c$ does not depend on $h$
since $\Pi^0_h$ reproduces piecewise constant matrix fields.  
It follows that $\Pi^0_h$ is bounded on $H^1(\Omega,\Sym)$.

The proof of the commutativity property $\div \Pi^0_h \tau = P_h \div \tau$ for all $\tau \in H^1(\Omega, \Sym)$ 
is similar to the one given in \cite{Arnold2005}, pp. 1423--1424. 
Next, let $R_h$ be a Clement interpolation operator \cite{Bernardi98}
which maps 
$L^2(\Omega,\Sym)$
into the subspace of $\Sigma_h$ of continuous matrix fields whose components 
are piecewise in $P_{1,1,1}$.
We have
\begin{equation*}
\|R_h \tau - \tau \|_j \leq c h^{m-j} \|\tau\|_m, \quad 0\leq j \leq 1,
\quad j \leq m 
\leq 2, 
\end{equation*}
with $c$ independent of $h$. The interpolation operator $\Pi_h$ is defined by
\begin{equation}
\Pi_h = \Pi^0_h (I - R_h) + R_h, \label{opdef}
\end{equation}
and as in \cite{Arnold2005}, one establishes its boundedness and the commutativity property.
We conclude that the pair $(\Sigma_h,V_h)$ is stable.

For the error analysis, 
we recall that the projection operator $P_h$ satisfies the error estimate
\begin{equation}
||P_h v - v||_0 \leq c h^m ||v||_m, \mbox{ for all } v \in H^m(\Omega), \quad 0
\leq m \leq 1, \label{l2pr}
\end{equation}
since $V_h$ contains piecewise constant vector fields. We then have the following theorem

\begin{thm}
\label{th1}
Let $(\sigma, u)$ and $(\sigma_h, u_h)$ be the unique solutions of \eqref{cf} and \eqref{cfD} respectively. Then
\begin{align*}
|| \div \sigma - \div \sigma_h ||_0 &\leq c h^m || \div \sigma ||_m, \quad
0 \leq m \leq 1, \\
 || \sigma - \sigma_h ||_0 &\leq c  h \|\sigma\|_1,  \\ 
 || u - u_h ||_0 &\leq c h||u||_{2}.
\end{align*} 
\end{thm}
\begin{proof}
We only outline the proof since it is essentially the same as the one in \cite{Arnold2002}.
Define $||.||_A$ by $||\tau||_A^2=\int_{\Omega} A \tau : \tau \ud x$. Then one shows that
$$
||\sigma - \sigma_h ||_A \leq ||\sigma - \Pi_h \sigma||_A.
$$
Since the norm $||.||_A$ is equivalent to the $L^2$ norm, we obtain using \eqref{pibound},
$$
||\sigma - \sigma_h ||_0 \leq c h \|\sigma\|_1.
$$
One also shows that $\div \sigma_h = P_h \div \sigma$.
The commutativity property and \eqref{l2pr} yields the error estimate on $\div \sigma$.
For the error estimate for the displacement, again from \cite{Arnold2002}, we have
$$
||P_h u -u_h||_0 \leq ||\sigma-\sigma_h||_0,
$$
so
\begin{align*}
||u-u_h||_0 & \leq ||u-P_h u||_0 + ||P_h u -u_h||_0 \leq c h ||u||_1 + c h \|\sigma\|_1 \leq c h ||u||_{2},
\end{align*}
since $A \sigma=\epsilon(u)$. 
\end{proof}

\section{Concluding remarks}
In three dimensions, the space of rigid body motions $\mathbb{T}$
is the space of linear polynomial functions of the form
$x\mapsto a+b \times x$ for some $a,b \in \mathbb{R}^3$. It is contained in $V(K)$ for each element $K$. 
Hence, as in \cite{Arnold2002,Arnold2003,Arnold2005,Arnold2008,Awanou2009}, we can take 
$V(K)=\mathbb{T}$ and the stress space as
\begin{equation}
\tau \in \Sigma(K):=\bigg\{\begin{pmatrix} \P_{311} & \P_{2 2 1} & \P_{212} \cr
\P_{2 2 1} & \P_{131} & \P_{122} \cr
\P_{212} & \P_{122} & \P_{113}
 \end{pmatrix}_{\Sym}, \div \tau \in \mathbb{T} \bigg\}.
\end{equation}
Since for $v \in \mathbb{T}, \epsilon(v)=0$, the degrees of freedom are the same as the ones in Lemma \ref{lemdof} with the last set
of degrees of freedom removed. The error analysis is the same. We note that the elements of \cite{Hu2007/08,Man2009,Chen2010} can be simplified as well
this way.

As noted for example in  \cite{Awanou10b}, high order mixed finite elements for elasticity usually require a high number of degrees of freedom
and typically require to find the dimension and a base for the space
$$
\{\tau \in \Sigma(K), \div \sigma=0, \sigma n = 0, \, \text{on} \, \partial \Omega \}.
$$
A possible approach is to follow the proof in the simplicial case \cite{Arnold2008}, but this goes well
beyond the scope of these remarks. 
Thus for that reason as well, we did not pursue here the construction of a discrete elasticity sequence.

\section{Acknowledgement}
The author thanks two anonymous referees for suggestions and questions leading to a better version of the paper. The author was supported in part by
NSF grant DMS-0811052 and the Sloan Foundation.


\end{document}